\documentclass[12pt,a4paper]{amsart}
\usepackage{amssymb,amsmath}

\numberwithin{equation}{section}
\newtheorem{theorem}{Theorem}[section]
\newtheorem{lemma}[theorem]{Lemma}
\newtheorem{proposition}[theorem]{Proposition}
\newtheorem{corollary}[theorem]{Corollary}

\theoremstyle{definition}
\newtheorem{definition}[theorem]{Definition}
\newtheorem{example}[theorem]{Example}
\newtheorem{notation}[theorem]{Notation}
\newtheorem{notation and remark}[theorem]{Notation and Remark}
\newtheorem{reminder}[theorem]{Reminder}
\newtheorem{remark}[theorem]{Remark}

\newcommand\Proj{\operatorname{Proj}}
\newcommand\Spec{\operatorname{Spec}}
\newcommand\Hom{\operatorname{Hom}}
\newcommand\Ext{\operatorname{Ext}}
\newcommand\Tor{\operatorname{Tor}}

\newcommand\depth{\operatorname{depth}}
\newcommand\codim{\operatorname{codim}}
\newcommand\beg{\operatorname{beg}}
\newcommand\reg{\operatorname{reg}}
\newcommand\sreg{\operatorname{sreg}}
\newcommand\Ker{\operatorname{\Ker}}
	
\newcommand\e{\operatorname{end}}

\newcommand\Ass{\operatorname{Ass}}

\newcommand\Nor{\operatorname{Nor}}

\newcommand\Sing{\operatorname{Sing}}

\newcommand\soc{\operatorname{soc}}
\newcommand\CM{\operatorname{CM}}
\newcommand\red{\operatorname{red}}
\newcommand\sat{\operatorname{sat}}
\newcommand\ch{\operatorname{Char}}

\begin{document}
\author[M. Brodmann \and P. Schenzel]{Markus Brodmann \quad Peter Schenzel}
\title[Curves of maximal regularity]{Projective Curves with maximal
regularity and applications to syzygies and surfaces}

\thanks{The second named author was partially supported by Swiss National 
Science Foundation (Project No. 20 - 111762 )}


\address{Universit\"at Z\"urich, Institut f\"ur Mathematik,
Winterthurer Str. 190, CH -- 8057 Z\"urich, Schwitzerland}

\email{markus.brodmann@math.uzh.ch}

\address{Martin-Luther-Universit\"at Halle-Wittenberg,
Institut f\"ur Informatik, D --- 06 099 Halle (Saale),
Germany}

\email{peter.schenzel@informatik.uni-halle.de}



\begin{abstract}
We first show that the union of a projective curve with one of its 
extremal secant lines  satisfies the linear general position principle 
for hyperplane sections. We use this to give an improved approximation 
of the Betti numbers of curves ${\mathcal C} \subset \mathbb P^r_K$ 
of maximal regularity with $\deg {\mathcal C} \leq 2r -3.$ In particular 
we specify the number and degrees of generators of the vanishing ideal 
of such curves. We apply these results to study surfaces 
$X \subset \mathbb P^r_K$ whose generic hyperplane section is a curve 
of maximal regularity. We first give a criterion for "an early decent 
of the Hartshorne-Rao function" of such surfaces. We use this criterion 
to give a lower bound on the degree for a class of these surfaces. Then, 
we study surfaces $X \subset \mathbb P^r_K$ for which 
$h^1(\mathbb P^r_K, {\mathcal I}_X(1))$ takes a value close to the 
possible maximum $\deg X - r +1.$ We give a lower bound on the degree
of such surfaces. We illustrate our results by a number of examples, 
computed by means of {\sc Singular}, 
which show a rich variety of occuring phenomena.
\end{abstract}

\maketitle

\section{Introduction}

The aim of this paper is to study projective curves of maximal regularity and applications
to certain types of projective surfaces. Much emphasis will be given to the 
computation of examples which illustrate the proved results. 

We begin with an investigation on curves of maximal regularity and their extremal
secant lines.
To be more precise, let ${\mathcal C} \subset \mathbb P^r_K$ be a non-degenerate irreducible projective curve in projective $r$-space over the algebraically closed field $K,$ with $r \geq 3.$ Let $d := \deg {\mathcal C}$ denote the degree of $\mathcal C.$ Assume that $d > r + 1$ and that $\mathcal C$ is of {\sl maximal regularity}, so that
\[
\reg {\mathcal C} = d - r +2.
\]
(Keep in mind that according to [GLP] we always have $\reg {\mathcal C} \leq  d - r +2 .$) In this extremal situation it is known that $\mathcal C$ is smooth and rational and has a $(d-r+2)$-secant line $\mathbb L.$ We fix such an {\sl extremal secant line} $\mathbb L = \mathbb P^1_K \subset \mathbb P^r_K,$ so  that
\[
\lambda ({\mathcal O}_{{\mathcal C} \cap \mathbb L}) = d - r + 2.
\]
In [BS2] we have shown that the scheme ${\mathcal C} \cup \mathbb L \subset \mathbb P^r_K$ plays a crucial r\^ole for the understanding of the curve $\mathcal C$ in particular its syzygetic behaviour. We first take up this idea and prove that in some sense the scheme ${\mathcal C} \cup \mathbb L$ behaves like an irreducible curve, namely (cf. Corollary \ref{2.6}):

\begin{theorem} \label{1.1} Let $r > 3.$ Then the generic hyperplane section $({\mathcal C} \cup \mathbb L) \cap \mathbb H \subset \mathbb H = \mathbb P^{r-1}_K$ of ${\mathcal C} \cup \mathbb L$ is a reduced scheme of $d+1$ points in linearly general position.
\end{theorem}

Now, on use of the $N_p$-Theorem of Green and Lazarsfeld \cite{GL} we then may approximate the Betti numbers of the homogeneous vanishing ideal $I \subset S := K[x_0,\ldots,x_r]$ of the curve $\mathcal C,$ provided $d$ is not to large (cf. Theorem \ref{3.3})

\begin{theorem} \label{1.2} Let $r > 3$ and $r + 2 \leq d \leq 2 r -3.$ Then, for all $i \in \{1,\ldots,r\}$ we have
\[
\Tor_i^S(K, S/I) \simeq K^{u_i}(-i-1) \oplus K^{v_i}(-i-2) \oplus K^{\binom{r-1}{i-1}}(-i-d+r-1)
\]
with $v_i = 0$ for $1 \leq i \leq 2r - d -2$ and $u_1 = \binom{r+1}{2} - d -1.$
\end{theorem}

In particular we get (cf. Corollary \ref{3.3} (b)):

\begin{corollary} \label{1.3} Under the hypothesis of \ref{1.2} the vanishing ideal $I \subset S$ of $\mathcal C$ is minimally generated by $\binom{r+1}{2} -d -1$ quadrics and one form of degree $d-r+2.$
\end{corollary}

In the second part of our paper we apply these results to certain surfaces. Our first aim is to study  a fairly technical
issue concerning
non-degenerate irreducible projective surfaces $X \subset
\mathbb P^r_K$ of degree $d \leq 2r - 4.$  Namely, we consider the
"least place at which the Hartshorne-Rao function $n \mapsto h^1(\mathbb P^r_K, {\mathcal I}_X(n))$
of $X$ definitively begins to decent", that is the invariant
\[
\delta(X) := \inf \{m \in \mathbb Z | h^1(\mathbb P^r_K, {\mathcal I}_X(n)) \leq
\max \{h^1(\mathbb P^r_K, {\mathcal I}_X(n-1)) -1, 0\} \text{ for all } m > n\}.
\]
It follows by Mumford's Lemma on the decent of the Hartshorne-Rao function (cf. \cite{M}) that
$\delta(X) \leq d-r+2.$ We are interested to find criteria which guarantee that this
inequality is strict. One sufficient condition surely would be that $X$ satisfies the
Regularity Conjecture of Eisenbud and Goto \cite{EG}, that is the inequality $\reg X
\leq d-r+3$ (cf. Lemma \ref{4.5}). Again by Mumford's Lemma one has $\delta (X) < d-r+2$ if the generic
hyperplane section of $X$ is not of maximal regularity (cf. Lemma \ref{4.5}). We shall
prove another criterion, namely (cf. Corollary \ref{4.6}):

\begin{corollary} \label{1.4} Let $r > 4$ and $r < d \leq 2r -4.$ Then
$
\delta(X) \leq  h^1(\mathbb P^r_K, {\mathcal I}_X(1)) + h^1(X, {\mathcal O}_X).
$
\end{corollary}

Next we study {\sl surfaces of maximal sectional regularity}, that is projective surfaces
$X \subset \mathbb P^r_K$ whose generic hyperplane section is a curve of
maximal regularity. For such surfaces $\reg X$ takes at least the value $d-r+3,$ the maximally
possible value conjectured by Eisenbud and G\^oto. Neverless there are surfaces $X$
satisfying $\reg X = d-r+3$ which are not of maximal sectional regularity (cf. Remark \ref{5.4}
(B)). Surfaces of maximal sectional regularity have "a lot" of extremal secant lines in 
the sense of Bertin \cite{Be1},
(cf. Proposition \ref{5.5} and Corollaries \ref{5.6}, \ref{5.7}). We prove the following bound
on the degree $d$ of these surfaces (cf. Theorem \ref{5.10}).

\begin{theorem} \label{1.5} Assume that $4 < r < d$ and that $X$ is of maximal sectional regularity and of arithmetic depth one. Then $\delta (X) \leq d-r+1$ implies that $d > 2r-5.$
\end{theorem}

It follows in particular (cf. Corollary \ref{5.13}):

\begin{corollary} \label{1.6} Let $4 < r < d,$ assume that $X \subset \mathbb P^r_K$
is Cohen-Macaulay, of maximal sectional regularity and of arithmetic depth one. Then $d > 2r - 5.$
\end{corollary}

It is well known that $h :=
h^1(\mathbb P^r_K, {\mathcal I}_X(1)) \leq d-r+1.$ In Section 6 we study surfaces $X$ with
"large $h$", that is with $d-r-1 \leq h \leq d-r+1.$ Using the concept of {\sl maximal
projecting surface $Y \subset \mathbb P^{r+h}_K$ for $X$} (cf. Reminder \ref{5.12}) and the
description of the possible types of surfaces $Z \subset \mathbb P^s_K$ of degree $s+1$ given
in \cite{B1} we prove (see Theorem 6.3 for more details):

\begin{theorem} \label{1.7} Let $4 < r < d.$ Then
\begin{itemize}
  \item[(a)] If $h = d-r+1,$ the surface $X$ is smooth and rational.
  \item[(b)] If $h = d-r,$ then either
       \begin{itemize}
         \item[(i)] $X$ is Cohen-Macaulay and the non-normal locus of $X$ is a straight line, or
         \item[(ii)] $X$ contains a single non-Cohen-Macaulay point and is normal outside this point.
        \end{itemize}
\item[(c)] If $h = d-r-1,$ we distinguish five different cases according to the type of the maximal projective surface $Y \subset \mathbb P^{r+h}_K$ for $X.$
\end{itemize}
\end{theorem}

An important issue of our paper are the examples contained in Section 7, which illustrate our 
results. These examples were computed by using the computer algebra system 
{\sc Singular} (cf. \cite{GP}).
\smallskip

\noindent{\bf Acknowledgement.} The main results of the present paper were found during a stay of the
authors at the "Mathematisches Forschungsinstitiut Oberwolfach" in the framework
of the "Research in Pairs" scheme. The authors express their gratitude toward this institution and the DFG for financial support and hospitality.

\section{Generic Hyperplane Sections of ${\mathcal C} \cup \mathbb L$}
Here let ${\mathcal C} \subset \mathbb P^r_K$ be a non-degenerate irreducible projective curve in projective $r$-space
of maximal regularity $\reg {\mathcal C} = d - r +2.$ Then $\mathcal C$ has a $(d-r+2)$-secant line $\mathbb L.$
In this section we show that generic hyperplane sections of the union ${\mathcal C} \cup \mathbb L$ are reduced schemes of points in linearly general position.

\begin{notation} \label{2.1} Let $\pi : \mathbb P^r_K \setminus \mathbb L \to \mathbb P^{r-2}_K$ denote a
linear projection with center $\mathbb L$ and let
\[
{\mathcal C}' := \overline{\pi ({\mathcal C} \setminus \mathbb L)} \subset \mathbb P^{r-2}_K
\]
be the closed image of ${\mathcal C} \setminus \mathbb L$ under $\pi.$ Observe that ${\mathcal C}' \subset \mathbb P^{r-2}_K$ is a reduced irreducible and non-degenerate curve.
\end{notation}

In characteristic 0, statement (a) of the following Lemma \ref{2.2} has an important extension for arbitrary varieties which admit
an extremal secant line (cf. \cite[Corollary 4.2]{Be2}).

\begin{lemma} \label{2.2} Let $r > 3.$  Then:
\begin{itemize}
\item[(a)] ${\mathcal C}' \subset \mathbb P^{r-2}_K$ is a rational normal curve.
\item[(b)] The map $\nu = \pi \upharpoonright : {\mathcal C} \setminus \mathbb L \to \mathbb P^{r-2}_K$ is an immersion.
\end{itemize}
\end{lemma}

\begin{proof}
(a): Let $d' := \deg {\mathcal C}'.$ As ${\mathcal C}' \subset \mathbb P^{r-2}_K$ is non-degenerate, it
suffices to show that $d' \leq r-2.$ But this is well known (cf. \cite[Corollary 4.2]{Be2} for example).

(b): It suffices to show that $\nu : {\mathcal C}\setminus \mathbb L \to {\mathcal C}'$ is an immersion. As ${\mathcal C}\setminus \mathbb L$ is affine and ${\mathcal C}' \simeq \mathbb P^1_K$ it suffices to show that $\nu : {\mathcal C}\setminus \mathbb L \to {\mathcal C}'$ is injective. Assume to the contrary that there are two points $p_1, p_1' \in {\mathcal C}\setminus \mathbb L$ such that $p_1 \not= p_1'$ but $\nu(p_1) = \nu(p_1') =: q_1 \in \mathbb P^{r-2}_K.$ Choose pairwise different points $q_2,\ldots, q_{r-2} \in \nu({\mathcal C}\setminus \mathbb L) \setminus \{q_1\}.$ As ${\mathcal C}' \subset \mathbb P^{r-2}_K$ is a rational normal curve by statement (a) we see that
\[
\mathbb H' := \langle q_1, \ldots, q_{r-2} \rangle \subset \mathbb P^{r-2}_K
\]
 is a hyperplane. Consider the hyperplane
\[
\mathbb H := \nu^{-1}(\mathbb H') = \nu^{-1}(\mathbb H') \cup \mathbb L.
\]
For each $i \in \{2,\ldots, r-2\}$ let $p_i \in {\mathcal C}\setminus \mathbb L$ be such that $\nu(p_i) = q_i,$ so that $p_1', p_1, p_2, \ldots, p_{r-2} \in ({\mathcal C} \cap \mathbb H) \setminus  \mathbb L$ are pairwise different points. The inequalities used in the proof of statement (a) show that $\# (({\mathcal C} \cap \mathbb H)_{\red} \setminus  \mathbb L) \leq r-2.$ So, we have a contradiction.
\end{proof}

\begin{lemma} \label{2.3} Let $r > 3$ and $p_1,\ldots, p_{r-1} \in {\mathcal C} \setminus \mathbb L$ be pairwise different points. Then $\langle p_1, \ldots, p_{r-1}\rangle \subset \mathbb P^r_K$ is of dimension $r-2$ and disjoint to $\mathbb L.$
\end{lemma}

\begin{proof} Let $\mathbb K := \langle p_1,\ldots, p_{r-1} \rangle.$ By Lemma \ref{2.2} and the notation
of \ref{2.1} the $r-1$ points $\pi(p_1), \ldots, \pi(p_{r-1}) \in {\mathcal C}'$ are pairwise different. As ${\mathcal C}' \subset \mathbb P^{r-2}_K$ is a rational normal curve we get $\langle \pi(p_1), \ldots, \pi(p_{r-2}) \rangle = \mathbb P^{r-2}_K.$  Therefore $\pi(\mathbb K \setminus \mathbb L) = \langle \pi(\mathbb K \setminus \mathbb L) \rangle \supseteqq \mathbb P^{r-2}_K,$ whence
\[
r - 2 \leq \dim \pi(\mathbb K \setminus \mathbb L) \leq \dim \mathbb K \leq r-2,
\]
thus $\dim \pi(\mathbb K \setminus \mathbb L) = \dim \mathbb K = r-2.$ The first inequality implies in particular that $\mathbb K \cap \mathbb L = \emptyset.$
\end{proof}

\begin{reminder} \label{2.4} (A) Let $s, d \in \mathbb N$ with $s < d$ and
let $p_1, \ldots, p_d \in \mathbb P^s_K$ be pairwise different points. Let $P = \{p_1, \ldots, p_d\}.$ We say that $p_1, \ldots, p_d$ are in {\sl linearly general position} if for all subsets $Q \subseteq P$ with $\# Q = s+1$ it follows that $\langle A \rangle = \mathbb P^s_K.$ This is equivalent to say that for all $Q \subseteq P$ we have $\dim \langle Q \rangle = \max \{ \# Q - 1 , s\}.$

(B) Let $p_{d+1} \in \mathbb P^s_K$ and assume that $p_1,\ldots,p_d$ are in linearly general position. Then $p_1,\ldots,p_d, p_{d+1}$ are in linearly general position if and only if for each set $Q \subseteq \{p_1,\ldots, p_d\}$ of cardinality $s$ we have $p_{d+1} \not\in \langle Q \rangle.$
\end{reminder}

\begin{proposition} \label{2.5} Let $r > 3$ and let $\mathbb H \subset \mathbb P^r_K$ be a hyperplane such that $\mathbb H \cap ({\mathcal C}\cap \mathbb L ) = \emptyset$ and such that ${\mathcal C}\cap \mathbb H \subset \mathbb H$ is a reduced subscheme of $d$ points in linearly general position. Then $({\mathcal C}\cup \mathbb L) \cap \mathbb H \subset \mathbb H$ is a reduced subscheme of $d+1$ points in linearly general position.
\end{proposition}

\begin{proof} We write $| {\mathcal C} \cap \mathbb H | = \{p_1,\ldots, p_d\}$ with pairwise different points $p_1,\ldots,p_d$ and $\{p_{d+1}\} = \mathbb H \cap \mathbb L.$ As $\mathbb H \cap ({\mathcal C}\cap \mathbb L) = \emptyset,$ the points $p_1, \ldots, p_d, p_{d+1}$ are pairwise different.

As $\deg {\mathcal C}\cup \mathbb L = d+1$ and
\[
| \mathbb H \cap ({\mathcal C}\cup \mathbb L) | \supseteq \{p_1, \ldots, p_d, p_{d+1}\}
\]
it follows that $\mathbb H \cap ({\mathcal C}\cup \mathbb L)$ is a reduced scheme of $d+1$ points. It remains to show that the points $p_1, \ldots, p_d, p_{d+1} \in \mathbb H = \mathbb P^{r-1}_K$ are in linearly general position. This follows by Lemma \ref{2.3} and Reminder \ref{2.4} (B).
\end{proof}

\begin{corollary} \label{2.6} Let $r > 3.$ For a generic hyperplane $\mathbb H \subset \mathbb P^r_K$ the subscheme $({\mathcal C}\cap \mathbb L) \cap \mathbb H \subseteq \mathbb H = \mathbb P^{r-1}_K$ is a reduced subscheme of $d+1$ points in linearly general position.
\end{corollary}

\begin{proof} This is immediate from Proposition \ref{2.5}.
\end{proof}

\section{Estimates of Betti Numbers}
\begin{notation} \label{3.1} (A) We consider the polynomial ring $S := K[x_0,\ldots,x_r]$ and write $\mathbb P^r_K = \Proj(S).$

(B) Let $I = I_{\mathcal C} \subset S$ denote the homogeneous vanishing ideal of ${\mathcal C} \subset \mathbb P^r_K$ and let $L \subset S$ denote the homogeneous vanishing ideal of the extremal secant line $\mathbb L.$ Finally, let $J := L \cap I \subset S$ denote the homogeneous vanishing ideal of ${\mathcal C} \cup \mathbb L \subset P^r_K.$

(C) If $m,n \in \mathbb Z$ are integers, we use the convention that $\binom{n}{m} = 0$ for $0 < n < m.$
\end{notation}

\begin{theorem} \label{3.2} Let $r > 3$ and assume that $r +2 \leq d \leq 2r - 3.$ For all $i \in \{1,\ldots, r+1\}$ set
\[
a_i = (d-r) \binom{r}{i} + \binom{r-1}{i-1} \text{ and }
c_i = (d-1) \binom{r-1}{i} - \binom{r-1}{i+1}.
\]
Then, for all $i \in \{1, \ldots, r\}$ we have
\[
\Tor_i^S(K, S/I) \simeq K^{u_i}(-i-1) \oplus K^{v_i}(-i-2) \oplus K^{\binom{r-1}{i-1}}(-i-d+r-1)
\]
with
\[
u_i \begin{cases} = \binom{r+1}{2} -d -1, & \text{for } i = 1, \\
                  = c_i - a_i, & \text{for } 2 \leq i \leq 2r-d-1, \\
                  \leq c_i, & \text{for } 2r-d \leq i \leq r;
    \end{cases}
\]
and
\[
v_i = \begin{cases} 0, & \text{for } 1 \leq i \leq 2r-d-2 \text{ and } i = r,\\
                    u_{i+1} + a_{i+1} -c_{i+1}, & \text{for } 2r-d - 1 \leq i \leq r-2,\\
                    d-r+1, & \text{for } i = r-1.
        \end{cases}
\]
\end{theorem}

\begin{proof} By our assumption we have $d \leq 2 r -1.$ So, by \cite[Proposition 3.5]{BS2} the ring $S/J$ is Cohen-Macaulay. Thus, \cite[Theorem 4.6]{BS2} yields that for all $i \in \{1,\ldots, r\}$ we have
\[
\Tor_i^S(K, S/I) \simeq K^{u_i}(-i-1) \oplus K^{v_i}(-i-2) \oplus K^{\binom{r-1}{i-1}}(-i-d+r-1)
\]
with (keep in mind the Notations \ref{3.1} (C))
\[
u_i \begin{cases} = \binom{r+1}{2} -d -1, & \text{for } i = 1, \\
                  \leq c_i, & \text{for } 2 \leq i \leq r;
    \end{cases}
\]
and
\[
v_i = \begin{cases} 0, & \text{for } 1 \leq i \leq 2r-d-2 \text{ and } i = r,\\
                    u_{i+1} + a_{i+1} -c_{i+1}, & \text{for } 2r-d - 1 \leq i \leq r-2,\\
                    d-r+1, & \text{for } i = r-1.
        \end{cases}
\]
(Unfortunately there is a misprint in the formula for $c_i$ given in \cite[Lemma 4.2]{BS2} and the formula should be as in the proof of that Lemma.)

It remains to show that $v_i = 0$ for all $i \leq 2r-d-2.$ So, let $\ell \in S_1$ be a generic linear form and set $\mathbb P^{r-1}_K = \mathbb H := \Proj (S/\ell S).$ Then, according to Corollary \ref{2.6} the scheme
\[
X = \Proj (S/(J + \ell S)) = ({\mathcal C} \cup \mathbb L) \cap \mathbb H \subset \mathbb H = \mathbb P^{r-1}_K
\]
is reduced and consists of
\[
d + 1 = 2(r-1) + 1 - (2r - d -2)
\]
points in linearly general position. So, by \cite[Theorem 1]{GL} the scheme $X \subset \mathbb P^{r-1}_K = \Proj(S/\ell S)$ satisfies condition $N_{2r-d-2}.$ As $S/J$ is a Cohen-Macaulay ring, $\ell $ is $S/J$-regular and $JS' \subset S' := S/\ell S$ is the homogeneous vanishing ideal of $X.$ As $X$ satisfies the condition $N_{2r-d-2}$ it follows  (with appropriate integers $d_1, \ldots, d_{2r-d-2}$)  that
\[
\Tor_i^{S'}(K, S'/JS') \simeq K^{d_i}(-i-1) \text{ for all } i \in \{1,\ldots, 2r-d-2\}.
\]
As $\ell \in S_1$ is $S/J$-regular we therefore get
\[
\Tor_i^S(K,S/J) \simeq K^{d_i}(-i-1) {\text{ for all }} i \in \{1,\ldots, 2r-d-2\}.
\]
According to \cite[Proposition 4.1]{BS2} we have
\[
\Tor_i^S(K, S/I) \simeq \Tor_i^S(K, S/J) \oplus K^{\binom{r-1}{i-1}}(-i-d+r-1)
\]
for all $i \in \{1,\ldots, r\}.$ In particular,
\[
\Tor_i^S(K, S/I) \simeq K^{d_i}(-i-1) \oplus K^{\binom{r-1}{i-1}}(-i-d+r-1)
\]
for all $i \in \{1, \ldots, 2r-d-2\}.$ As $d \geq r+2$ it follows that $v_i = 0$ for $1 \leq i \leq 2r-d-2.$
\end{proof}

\begin{corollary} \label{3.3} Let $r > 3$ and assume that $r+2 \leq d \leq 2r-2.$ Then:
\begin{itemize}
\item[(a)] The vanishing ideal $I \subset S$ of ${\mathcal C} \subset \mathbb P^r_K$ is minimally
generated by $\binom{r+1}{2} -d-1$ quadrics, at most $(d-1)\binom{r}{2} +r -1$ cubics  and one form of degree $d-r+2.$
\item[(b)] If $ d \leq 2r -3,$ there are no cubics in a minimal generating set of the vanishing ideal $I
\subset S.$
\end{itemize}
\end{corollary}

\begin{proof} $(a):$ This follows from \cite[Proposition 3.5 and Theorem 4.6]{BS2}.

$(b):$ By Theorem \ref{3.2} we have
\[
\Tor_1^S(K, S/I) \simeq K^{u_1}(-2) \oplus K(-d+r-2)
\]
with $u_1 = \binom{r+1}{2}-d-1.$ This proves our claim.
\end{proof}

\section{Surfaces of Degree $\leq 2r - 4$}
\begin{notation} \label{4.1} (A) Let $X \subset \mathbb P^r_K = \Proj S$ with
$S = K[x_0,\ldots, x_r],$ be an irreducible, reduced non-degenerate projective
variety of degree $d > r \geq 4.$ By $I$ we denote the homogeneous vanishing
ideal of $X$ in $S$ and by $A$ the homogeneous coordinate ring $S/I$ of $X.$

(B) We write $S_+$ for the irrelevant ideal $\oplus_{n > 0}S_n = (x_0,\ldots,x_n)S$
of $S.$ If $M$ is a graded $S$-module and $i \in \mathbb N_0,$ we write $H^i(M)$
for the $i$-th local cohomology module of $M$ with respect to $S_{+},$ furnished with its natural grading.
If the graded $S$-module $M$ is finitely generated we write $h^i_M(n)$ for the vector-space
dimension $\dim_K H^i(M)_n$ of the $n$-th graded component of $H^i(M).$
\end{notation}

\begin{remark} \label{4.2} Keep the above notations and hypothesis. Let ${\mathcal I}_X \subset
{\mathcal O}_{\mathbb P^r_K}$ be the sheaf of vanishing ideals of $X.$ Then, the well known relations
between local cohomology and sheaf cohomology yield $h^{i+1}_A(n) = h^i(X, {\mathcal O}_X(n))$
for all $i > 0$ and all $n \in \mathbb Z$ and $h^i_A(n) = h^i(\mathbb P^r_K, {\mathcal I}_X(n))$
for all $i \not= 0, r$ and all $n \in \mathbb Z.$
\end{remark}

The following technical result is a generalization of statements shown in \cite[4.4, 4.5, 4.6]{B1}
for the special case $d = r+1.$

\begin{proposition} \label{4.3} Let $6 \leq r+1 \leq d \leq 2r-4$ and let $(f,g) \in S_1^2$ be a
pair of generic linear forms. Then:
\begin{itemize}
\item[(a)] $H^0(A/(f,g)A)$ is generated by homogeneous elements of degree 2.
\item[(b)] $H^1(A/(f,g)A)_n = 0$ for all $n \geq 2.$
\item[(c)] For all $(\lambda, \mu) \in K^2\setminus \{(0,0)\}$ and all $n \geq 2$ we have
\[
h^1_{A/(\lambda f + \mu g)A}(n) \leq \max \{h^1_{A/(\lambda f + \mu g)A}(n-1) - 1, 0\}.
\]
\item[(d)] For all $m \in \mathbb N$ and $n \geq h^1_A(m) +h^2_A(m-1) +m$ we have
\[
h^1_A(n) \leq \max \{h^1_A(n-1) -1, 0\}.
\]
\end{itemize}
\end{proposition}

\begin{proof} (a): By our choice of $f$ and $g$ the scheme $Z = \Proj(A/(f,g)A) \subset
\mathbb P^{r-2}_K$ consists of $d = 2(r-2) + 1 - (2r-d-3)$ points in linearly general
position. So, by \cite[Theorem 1]{GL} this scheme $Z$ satisfies
the property $N_{2r-d-3},$ with $2r-d-3 \geq 1.$ Therefore the vanishing ideal
\[
I_Z = (f,g,I)^{\sat}/(f,g)S \subset S/(f,g)S
\]
of $Z$ is generated by homogeneous elements of degree 2. So, the same is true for
\[
I_Z/((f,g,I)/(f,g)S) \simeq (f,g,I)^{\sat}/(f,g, I) \simeq H^0(A/(f,g)A).
\]

(b): The scheme $Z = \Proj (A/(f,g)A) \subset \mathbb P^{r-2}_K$ consists of $d$ points
in semi-uniform position. So, by \cite[Lemma 2.4 (a)]{BS1} we have:
\[
h^1_{A/(f,g)A}(n) \leq \max \{d-1 -n(r-2), 0\}
\]
for all $n \geq 0.$ As of $d \leq 2r - 4$ we get our claim.

(c): Without loss of generality we may assume that $\mu \not= 0.$ We set $\ell := \lambda f + \mu g,$
so that $(f, \ell)A = (f,g)A.$ We put $E := A/\ell A$ and $\bar{E} := E/H^0(E).$
Our first claim is, that $f$ is $\bar{E}$-regular. Indeed, otherwise we would find some
$\mathfrak p \in \Ass \bar{E} = \Ass E \setminus \{S_+\}$ with $f \in \mathfrak p.$
As $f$ and $\ell$ are $A$-regular, it would follow $\mathfrak p \in \Ass A/fA \setminus \{S_+\}.$
By $gA \subset (f, \ell) A$ and $f, \ell \in \mathfrak p$ we also would have $g \in \mathfrak p$
and so $f,g$ would not form a regular sequence with respect to $A_{\mathfrak p}.$ This would
contradict the genericity of the pair $(f,g) \in S^2_1.$

So, $f$ is $\bar{E}$-regular and we get a short exact sequence
\[
0 \to \bar{E}(-1) \stackrel{f}{\to} \bar{E} \to \bar{E}/f \bar{E} \to 0.
\]
In view of the natural isomorphisms $H^1(\bar{E}) \simeq H^1(E)$ and
\[
H^1(\bar{E}/f\bar{E}) \simeq H^1(E/fE) \simeq H^1(A/(\ell, f)A) \simeq H^1(A/(f,g)A)
\]
statement (b) shows that for all $n \geq 2$ there is an exact sequence
\[
0 \to H^0(\bar{E}/f\bar{E})_n \to H^1(E)_{n-1} \to H^1(E)_n \to 0.
\]
Moreover there is an epimorphism of graded $S$-modules $H^0(A/(f,g)A) \twoheadrightarrow H^0(\bar{E}/f\bar{E}).$
So, by statement (a) the module $H^0(\bar{E}/f\bar{E})$ is generated by homogeneous
elements of degree 2. But now, the above sequences and the fact that $H^1(E)_n = 0$ for all $n \gg 0$ imply that for all $n \geq 2$ we have
$h^1_E(n) \leq \max \{h^1_E(n-1) -1, 0\}.$ This is our claim.

(d): Let $(\lambda, \mu) \in K^2 \setminus \{(0,0)\}$ and set $h = \lambda f + \mu g.$ If we apply cohomology to the short exact sequence
\[
0 \to A(-1) \stackrel{h}{\to} A \to A/hA \to A
\]
we get exact sequences
\[
H^1(A)_{n-1} \stackrel{h}{\to} H^1(A)_n \to H^1(A/hA) \to H^2(A)_{n-1}.
\]
Applying this with $n = m,$ we see that
\[
h^1_{A/hA}(m) \leq h^1_A(m) + h^2_A(m-1).
\]
So, by statement (c) we get $h^1_{A/hA}(n) = 0$ and hence an epimorphism $H^1(A)_{n-1} \twoheadrightarrow H^1(A)_n$ for all $n \geq h^1_A(m) + h^2_A(m-1) + m$ and all pairs $(\lambda, \mu) \in K^2\setminus \{(0,0)\}.$ By \cite[Lemma 3.2]{B1} we may conclude that for all $n \geq h^1_A(m) + h^2_A(m-1) + m$ we have
\[
h^1_A(n) \leq \max \{h^1_A(n-1) - 1,0\}.
\]
\end{proof}

\begin{notation and remark} \label{4.4} (A) Later we shall be interested in the place, at which the Hartshorne-Rao function $n \mapsto h^1_A(n) = h^1(\mathbb P^r_K, {\mathcal I}_X(n))$ definitively descents, that is in the invariant
\[
\delta (X) := \inf \{m \in \mathbb Z | h^1_A(n) \leq \max \{h^1_A(n-1) - 1, 0\} \text{ for all } n > m\}.
\]
(B) If $T = \oplus_{n \in \mathbb Z} T_n$ is a graded $S$-module, we define the {\it beginning} and the {\it end} of $T$
by
\[
\beg (T) := \inf \{n \in \mathbb Z | T_n \not= 0\}, \e (T) = \sup \{n \in \mathbb Z | T_n \not= 0\}.
\]
\end{notation and remark}

\begin{lemma} \label{4.5} Let $\mathbb H := \mathbb P^{r-1}_K \subset \mathbb P^r_K$ be a hyperplane such
that the intersection curve ${\mathcal C} := X \cap \mathbb H \subset \mathbb H$ is reduced and irreducible.
Then:
\begin{itemize}
  \item[(a)] $\reg {\mathcal C} \leq \min \{ \reg X , d-r+3\}.$
  \item[(b)] $\delta(X) \leq \min \{ \reg {\mathcal C} -1, \e H^1(A) \} \leq \reg X -2. $
\end{itemize}
\end{lemma}

\begin{proof} (a): The relation $\reg {\mathcal C} \leq \reg X$
is well known. Moreover,
by \cite{GLP} we have
\[
\reg {\mathcal C} \leq \deg {\mathcal C} -(r-1) + 2 = d-r+3.
\]
(b): By Mumford's Lemma on the descent of the Hartshorne-Rao  function
\[
n \mapsto h^1_A(n) = h^1(\mathbb P^r_K, {\mathcal I}_X(n))
\]
(cf. \cite[page 102, statement \#']{M}) we have
\[
h^1_A(n) \leq \max \{h^1_A(n-1) - 1, 0\} {\text{ for all }}
n > \reg {\mathcal C} -1,
\]
 so that $\delta (X) \leq \reg {\mathcal C} - 1.$ Clearly $\delta (X) \leq
\e H^1(A) \leq \reg X -2.$
\end{proof}

\begin{corollary} \label{4.6} Let $4 < d < 2r -4.$  Then $\delta (X) \leq  \min \{
d-r+2, h^1_A(1) + h^2_A(0) \}.$
\end{corollary}

\begin{proof} Apply statements (a) and (b) of Lemma \ref{4.5} and
Proposition \ref{4.3} (d) with $m = 1.$
\end{proof}

We close this section with another result which helps to pave the way for our investigations
of surfaces of maximal sectional regularity.

\begin{notation} \label{4.7} (A) We introduce the invariant
\[
e(X) := \sum_{x \in X, x \text{ closed}} 
\lambda_{{\mathcal O}_{X,x}}(H^1_{{\mathfrak m}_{X,x}}({\mathcal O}_{X,x})) (< \infty)
\]
which counts the number of non-Cohen-Macaulay points on $X$ in a weighted way.

(B) By $\sigma (X)$ we denote the {\sl sectional genus} of $X,$ that is the arithmetic genus of the generic hyperplane section of $X.$ So,  we have that
\[
\sigma (X) = h^2_{A/fA}(0),
\]
for a generic linear form $f \in S_1.$

(C) We denote the {\it normal locus}, the {\it Cohen-Macaulay locus} and the
{\it singular locus} of $X$ respectively by $\Nor (X), \CM(X)$ and $\Sing (X).$
\end{notation}

\begin{remark} \label{4.8} (A) According to \cite[Proposition 5.9]{B0} we have
\[
e(X) \leq h^2_A(n-1) \leq \max \{ e(X), h^2_A(n) -1\} {\text{ for all }} n \leq 0.
\]

(B) As $X$ is a surface, we have $\# (X \setminus \Nor(X)) < \infty$ if and only if $\# \Sing (X) <
\infty.$

(C) Moreover, by Bertini's Theorem the generic hyperplane section ${\mathcal C} = X \cap \mathbb P^{r-1}_K$
is smooth, if and only if $\Sing (X)$ is a finite set.

(D) For a generic linear form $f \in S_1$ we have $h^1_{A/fA}(0) = 0$ and
$h^3_{A/fA}(0) = 0.$ As $\sigma (X) = h^2_{A/fA}(0),$ the
short exact sequence $0 \to A(-1) \to A \to A/fA \to 0$ yields:
\[
\sigma (X) = h^2_A(0) - h^2_A(-1) - ( h^3_A(0) - h^3_A(-1)).
\]
In particular, $\sigma (X)$ is the sectional genus of the polarized pair $(X, {\mathcal O}_X(1))$
in the sense of Fujita \cite{F}.
\end{remark}

\begin{proposition} \label{4.9} Assume that $\sigma (X) = 0.$ Let $f \in S_1$ be a generic linear form and
set ${\mathcal C} = \Proj A/fA.$ Then:
\begin{itemize}
  \item[(a)] \begin{itemize}
  \item[(i)] $h^2_A(-1) = h^2_A(0),$
  \item[(ii)]  $h^2_A(0) - h^2_A(1) = h^1_{A/fA}(1) - h^1_A(1) \geq 0,$
  \item[(iii)] $h^2_A(n) \leq \max \{ 0, h^2_A(n-1) -1 \}$ for all $n \geq 2,$
  \item[(iv)] $h^3_A(n) = 0$ for all $n \geq -1.$ \end{itemize}
  \item[(b)] ${\mathcal C} \simeq \mathbb P^1_K$ if and only if $\# \Sing (X) < \infty.$
  \item[(c)] If $4 < r < d \leq 2r - 2$ and $\# \Sing (X) < \infty,$ then $h^2_A(n) = e(X)$ for all $n \leq 0.$
\end{itemize}
\end{proposition}

\begin{proof} (a): We have $h^2_{A/fA}(0) = \sigma (X) = 0,$ so
that $h^2_{A/fA}(n) = 0$ for all $n \geq 0.$ As $h^1_{A/fA}(n) = 0$ for all
$n \leq 0$ and $h^3_A(n) = 0$ for all $n \gg 0$ the statements
(i), (ii), and (iv) follow immediately from the short exact
sequence $0 \to A(-1) \to A \to A/fA \to 0.$ Statement (iii) follows
from \cite[Proposition 3.5 (b)]{B2}.

(b): As ${\mathcal C}$ is of arithmetic genus 0, it is smooth if and
only if it is isomorphic to $\mathbb P^1_K.$ Now, we conclude by Remark
\ref{4.8} (b).

(c): According to \cite[Proposition 3.8 (b)]{B2} we have $h^2_A(n) = e(X)$ for all $n < 0.$
We thus get our claim by statement (a) (i).
\end{proof}

\section{Surfaces of maximal sectional regularity}
We keep the notations and hypothesis of the previous section.

\begin{definition} \label{5.1} We define the {\sl sectional regularity} $\sreg X$ as the
least value of $\reg X \cap \mathbb H,$ where $\mathbb H = \mathbb P^{r-1}_K \subset
\mathbb P^r_K$ runs through all hyperplanes of $\mathbb P^r_K.$ Thus we may write
\[
\sreg X = \min \{ \reg (\Proj (A/fA)) | f \in S_1 \setminus \{0\} \}.
\]
\end{definition}

\begin{remark} \label{5.2} (A) For all $f \in S_1 \setminus \{0\}$ we have
\[
\reg (\Proj (A/fA)) = \max \{ \e H^1(A/fA) +2, \e H^2(A/fA) +3\}
\]
and so the short exact sequences
\begin{equation*}
\begin{gathered}
H^1(A)_{n-1} \stackrel{f}{\to} H^1(A)_n \to H^1(A/fA)_n \to H^2(A)_{n-1} \stackrel{f}{\to} H^2(A)_n \\
\to H^2(A/fA)_n \to H^3(A)_{n-1} \stackrel{f}{\to} H^3(A)_n
\end{gathered}
\end{equation*}
yield that the set
\[
\{ f \in S_1 \setminus \{0\} | \reg (\Proj (A/fA)) = \sreg X\}
\]
is dense and open in $S_1.$ In particular, for a generic hyperplane section $\mathcal C
= X \cap \mathbb P^{r-1}_K$ of $X$ we have $\reg {\mathcal C} = \sreg X.$ So that $\sreg X$ is the
regularity of the generic hyperplane section of $X.$

(B) As the generic hyperplane section $\mathcal C = X \cap \mathbb P^{r-1}_K$ of $X$ is reduced
and irreducible it follows 
\begin{equation*}
\begin{gathered}
\sreg X \leq \reg X, {\text{ with equality if }} X {\text{ is of arithmetic depth}} > 1;\\
\sreg X \leq d-r+3, {\text{ and }} \delta (X) \leq \sreg X -1.
\end{gathered}
\end{equation*}
\end{remark}

\begin{definition} \label{5.3} In view of Remark \ref{5.2} (B) it makes sense to say that the surface
$X \subset \mathbb P^r_K$ is of {\sl maximal sectional regularity} if $\sreg X = d-r+3.$ It is equivalent
to say that the generic hyperplane section $\mathcal C = X \cap \mathbb P^{r-1}_K \subset \mathbb P^{r-1}_K$
satisfies $\reg {\mathcal C} = d-r+3$ and thus is a curve of maximal regularity.
\end{definition}

\begin{remark} \label{5.4} (A) Assume that the surface $X \subset \mathbb P^r_K$ is of maximal
sectional regularity. Assume that $d = \deg X > r.$ Then, the generic hyperplane section
curve ${\mathcal C} = X \cap \mathbb P^{r-1}_K \subset \mathbb P^{r-1}_K$ is of maximal
regularity and of degree $d > r-1.$ So by \cite[Table on p. 505]{GLP} the curve $\mathcal C$
is smooth and rational. Moreover, it follows (cf. Proposition \ref{4.9} (b))
\[
{\mathcal C} \simeq \mathbb P^1_K, \sigma (X) = 0, \text{ and } \# \Sing (X) < \infty.
\]
(B) If a surface $X \subset \mathbb P^r_K$ of degree $d$ is of maximal sectional regularity, it
satisfies $\reg X \geq d-r+3$ by Remark \ref{5.2} (B). On the other hand there are examples showing that $\reg X =
d-r+3$ need not imply that $X$ is of maximal sectional regularity (cf. Example \ref{7.1} (A)).
So, being of maximal sectional regularity is strictly stronger than having the maximal
regularity conjectured by Eisenbud and G\^oto.
\end{remark}

Let $X \subset \mathbb P^r_K$ be as above. According to Bertin \cite{Be1} a line $\mathbb L = \mathbb P^1_K
\subset \mathbb P^r_K$ is called an {\sl extremal secant line to } $X$ if $\lambda({\mathcal O}_{X \cap \mathbb L}) = d-r+3.$ Concerning the relation between the condition that $X$ is of maximal sectional regularity
and that $X$ has an extremal secant line we have the following result.

\begin{proposition} \label{5.5} Let $4 < r < d,$ let $\mathbb L \subset \mathbb P^r_K$ be a line and let
$\mathcal H$ be the linear system of all hyperplanes $\mathbb H \subset \mathbb P^r_K$ with $\mathbb L \subset
\mathbb H.$ Consider the following four statements:
\begin{itemize}
\item[(i)] For a generic $\mathbb H \in {\mathcal H}$ the curve ${\mathcal C} = X \cap \mathbb H \subset
\mathbb H$ is reduced, irreducible, of regularity $d-r+3$ and $\mathbb L$ is an extremal secant line to $\mathcal C.$
\item[(ii)] There is some $\mathbb H \in {\mathcal H}$ such that the curve ${\mathcal C} = X \cap \mathbb H \subset
\mathbb H$ satisfies the requirements of statement (i).
\item[(iii)] $\mathbb L$ is a secant line to $X$ with $\mathbb L \cap X \subseteq \CM(X).$
\item[(iv)] $\mathbb L \not\subseteq X, \mathbb L \cap X \subseteq \CM(X)$ and $\lambda ( {\mathcal O}_{X \cap \mathbb L})
\geq d-r+3.$
\end{itemize}
Then
\begin{itemize}
\item[(a)] We have the implications (i) $\Rightarrow$ (ii) $\Rightarrow$ (iii) $\Rightarrow$  (iv).
\item[(b)] If $\ch K = 0,$ then all of the four statements are equivalent.
\end{itemize}
\end{proposition}

\begin{proof} (a): It suffices to show that the implication "(ii) $\Rightarrow$ (iii)" holds. So, let
$\mathbb H \in {\mathcal H}$ be as in statement (ii). Then ${\mathcal C} = X \cap \mathbb H \subset \mathbb H$
is a reduced and irreducible curve with $\deg {\mathcal C} = d$ and $\reg {\mathcal C} =
\lambda ({\mathcal O}_{X \cap \mathbb L}) = d-r+3.$ As $X$ is a surface and $\mathcal C$ is Cohen-Macaulay
and locally cut out by one equation from $X,$ we must have $X \cap \mathbb H = {\mathcal C} \subseteq \CM(X),$
whence $X \cap \mathbb L \subseteq \CM(X).$ Moreover
\[
\lambda ({\mathcal O}_{X \cap \mathbb L}) = \lambda ({\mathcal O}_{(X \cap \mathbb H) \cap \mathbb L})
= \lambda ({\mathcal O}_{{\mathcal C} \cap \mathbb L}) = d-r+3.
\]

(b): It suffices to prove the implication "(iv) $\Rightarrow$ (i)". So let $\mathbb L$ be as in
statement (iv). As $\# (\mathbb L \cap X) < \infty$ and $\ch K = 0$ it follows by Bertini Theorems
that the curve ${\mathcal C} = X \cap \mathbb H \subset \mathbb H$ is reduced outside $\mathbb L \cap X$
and irreducible for generic $\mathbb H \in {\mathcal H}.$  As $\mathbb L \cap X \subset \CM(X)$ and $\mathcal C$
is locally cut out from $X$ by one equation, it follows that $\mathcal C$ is also reduced in all
points $x \in \mathbb L \cap X.$ So $\mathcal C$ is reduced and irreducible at all. Moreover $\reg {\mathcal C}
\geq \lambda ({\mathcal O}_{{\mathcal C} \cap \mathbb L}) = \lambda ({\mathcal O}_{X \cap \mathbb L}) \geq
d-r+3 = \deg {\mathcal C} -(r-1) +2 \geq \reg {\mathcal C}$ (cf. \cite{GLP}). This proves our claim.
\end{proof}

\begin{corollary} \label{5.6} Let $4 < r < d.$ If there is a hyperplane $\mathbb H \subset \mathbb P^r_K$
such that the curve ${\mathcal C} = X \cap \mathbb H \subset \mathbb H$ is reduced and irreducible and of
regularity $d-r+3,$ then $X$ admits an extremal secant line $\mathbb L \subset \mathbb P^r_K$ such that
$X \cap \mathbb L \subseteq \CM(X).$ If $\ch K = 0,$ the converse is true also.
\end{corollary}

\begin{proof} Clear from Proposition \ref{5.5} and the fact that a curve of maximal regularity of degree $d$ in
$\mathbb P^{r-1}_K$ has an extremal secant line (cf. \cite{GLP}).
\end{proof}

\begin{corollary} \label{5.7} Let $4 < r < d.$ If $X$ is of maximal sectional regularity, it admits
an extremal secant line $\mathbb L$ such that $X \cap \mathbb L \subseteq \CM(X).$
\end{corollary}

\begin{proof} Clear from Corollary \ref{5.6}.
\end{proof}

So, in this section, we are actually interested in surfaces $X$ having not only one extremal secant line,
but "many of them". Moreover we wish not to restrict the type of singularities of $X$ nor the characteristic
of the base field $K.$ We thus cannot make use of the classification given in \cite{Be1} and \cite{Be2}.

The technical key result of the present section is Proposition \ref{5.9} below. The natural aim
to prove statement (a) of this Proposition would be to apply the Socle Lemma of Huneke and Ulrich
(cf. \cite[Corollary (3.11) (i)]{HU}). But this would mean that we had to assume that the base
field $K$ is of characteristic $0.$ As we prefer a characteristic free approach, we shall
attack the proof of Proposition \ref{5.9} in a slightly different way. We namely first prove
the following lemma, which follows easily from a result of Kreuzer (cf. \cite{mK}).

For a graded $S$-module $M$ let
\[
\soc M := 0 :_M S_+ = \Hom_S( K, M)
\]
denote the {\sl socle} of $M.$

\begin{lemma} \label{5.8} Let $\ell \in S_1$ be a generic linear form and let $T$ be a finitely generated graded $S$-module. Then for each integer $n < \beg \soc (T/\ell T)$
we have $(0 :_T \ell)_n \subset \ell T_{n-1}.$
\end{lemma}

\begin{proof} Let $n < \beg \soc (T/\ell T)$ and $t \in
(0 :_T \ell)_n.$ Then $\ell t = 0.$ As $\ell \in S_1$ is generic and $K$ is infinite, we may
apply \cite[Corollary (1.2) (b)]{mK} to the finitely generated graded $S$-module $T$ and get that $\ell' t \in \ell T = 0$ for all $\ell' \in S_1.$ Therefore $t + \ell T \in (\soc
(T/\ell T))_n = 0,$ thus $t \in \ell T.$
\end{proof}

\begin{proposition} \label{5.9} Assume that $4 < r < d$ and the surface $X$ is of maximal sectional regularity. Let $\ell \in S_1$ be a generic linear form and set $U := (0:_{H^1(A)} \ell )(-1).$ Then:
\begin{itemize}
\item[(a)] If $d \leq 2r-4,$ then $U$ is minimally generated by at most $(d-1)\binom{r-1}{2} +r-2$ forms of degree 3 and at most one form of degree $d-r+3.$
\item[(b)] If $d \leq 2r-5,$ then $U$ is generated by at most one single form of degree $d-r+3.$
\end{itemize}
\end{proposition}

\begin{proof} Let $\mathbb H := \Proj S/\ell S = \mathbb P^{r-1}_K,$ so that \[
{\mathcal C} := X \cap \mathbb H = \Proj A/\ell A \subset \mathbb H = \mathbb P^{r-1}_K
\]
is a curve of maximal regularity and degree $d$ in $\mathbb P^{r-1}_K.$ If we apply cohomology to the short
exact sequence
\[
0 \to A(-1) \stackrel{\ell}{\longrightarrow} A \to A/\ell A \to 0
\]
we get an isomorphism of graded $S$-modules $H^0(A/\ell A) \simeq U.$ Now, let
\[
I_{\mathcal C} := (I +\ell S)^{\sat}/\ell S \subset S/\ell S
\]
be the homogeneous vanishing ideal of $\mathcal C.$ Then the isomorphisms
\[
I_{\mathcal C}/((I +\ell S)/\ell S) \simeq (I+\ell S)^{\sat}/(I + \ell S) = H^0(S/(I + \ell S)) \simeq H^0(A/\ell A)
\]
show that there is an epimorphism of graded $S$-modules $I_{\mathcal C} \twoheadrightarrow U.$ As
\[
\deg {\mathcal C} = d \leq 2r - 4 = 2(r-1) -2
\]
we get from Corollary \ref{3.3} that $I_{\mathcal C} \subset S/\ell S$ is minimally generated by
$\binom{r}{2} -d-1$ quadrics, at most $(d-1) \binom{r-1}{2} +r-2$ cubics and one form of degree
$d-r+3.$ Moreover the cubics are not needed if
$\deg {\mathcal C} = d \leq 2(r-1) -3 = 2r-5.$ In view of the previous epimorphism
$I_{\mathcal C} \twoheadrightarrow U$ it therefore remains to show that $\beg U \geq 3$ or --
equivalently -- that $\beg (0:_{H^1(A)} \ell) \geq 2.$
Observe that $H^1(A/\ell A) \simeq H^1(A_{\mathcal C}),$ where
$A_{\mathcal C} = (S/\ell S)/I_{\mathcal C} \simeq (A/\ell A)/H^0(A/\ell A)$ is the homogeneous
coordinate ring of ${\mathcal C} \subset \mathbb P^{r-1}_K.$ Now, applying \cite[Proposition 3.5 and Theorem 3.3]{BS2}
to $\mathcal C,$ we get $\soc H^1(A_{\mathcal C}) \simeq K(r-d-1)$ and hence $\soc H^1(A/\ell A) \simeq K(r-d-1).$
Thus $\soc H^1(A)/\ell H^1(A) \simeq
K(r-d-1),$ so that $\beg (\soc H^1(A)/\ell H^1(A)) = d-r+1 \geq 2.$

Observe that $\beg H^1(A) \geq 1.$ Assume now, that $\beg U \leq 2.$ Then, there is an element
$u \in U_2 \setminus \{0\} = (0:_{H^1(A)} \ell)_1 \setminus \{0\} = H^1(A)_1 \setminus \{0\}.$
According to Lemma \ref{5.8} we get $u \in \ell H^1(A)_0 = \ell \cdot 0 = 0,$ a contradiction.
\end{proof}

In the next result we use the invariant $\delta (X) $ introduced in Notation \ref{4.4}.

\begin{theorem} \label{5.10} Assume that $4 < r < d$ and the surface $X \subset \mathbb P^r_K$ is of maximal sectional regularity, of arithmetic depth one and satisfies $\delta (X) \leq d-r+1.$ Then $d > 2r-5.$
\end{theorem}

\begin{proof} Assume that all the hypotheses are satisfied and that $d \leq 2r-5.$ Let $\ell \in S_1$ and $U$ be as in Proposition  \ref{5.9}. Then statement (b) of this Proposition yields that $U$ is generated by a form of degree $d-r+3,$ or vanishes. We consider the exact sequences
\[
0 \to U_n \to H^1(A)_{n-1} \stackrel{\ell}{\to} H^1(A)_n
\]
for all $n \in \mathbb Z.$

As $U_n = 0$ for all $n \leq d-r+2$ we have $h^1_A(n) \geq h^1_A(n-1)$ for all of these $n.$ As $\delta (X) \leq d-r+1$ we also have $h^1_A(n) \leq \max \{h^1_A(n-1) - 1, 0\}$ for all $n \geq d-r+2.$ Both statements together yield $H^1(A) = 0,$ and this contradicts the hypothesis that $X$ is of arithmetic depth one.
\end{proof}

\begin{corollary} \label{5.11} Let $4 < r < d$ and assume that the surface $X \subset \mathbb P^r_K$ is of maximal sectional regularity and of arithmetic depth one. Assume that $h^1_A(1) + h^2_A(0) \leq d-r+1$ or
$\e (H^1(A)) \leq d-r+2.$ Then $d > 2r -5.$
\end{corollary}

\begin{proof} This is immediate by Lemma \ref{4.5} (b), Proposition \ref{4.6} and Theorem \ref{5.10}.
\end{proof}

\begin{reminder} \label{5.12} (A) Let $h \in \mathbb N_0.$ A non-degenerate irreducible projective
surface $Y \subset \mathbb P^{r+h}_K$ is called a {\it projecting surface for} $X,$ if there is
a linear projection $\pi :  \mathbb P^{r+h}_K \setminus \mathbb P^{h-1}_K \twoheadrightarrow \mathbb P^r_K$ whose
center $\mathbb P^{h-1}_K$ is disjoint to $Y$ and such that by restricting $\pi$ we get an isomorphism
$\pi \upharpoonright : Y \stackrel{\simeq}{\to} X.$ Such projecting surfaces $Y \subset \mathbb P^{r+h}_K$ for $X$
exist if and only if $h \leq h^1(\mathbb P^r_K, {\mathcal I}_X(1)) = h^1_{A}(1).$ If $h = h^1_{A}(1),$ we
speak of {\it maximal projecting surfaces for } $X.$ Moreover, we call $h^1_A(1)$ the {\sl linear deficiency} of $X.$

(B) Now, let $h = h^1_A(1)$ and let $Y \subset \mathbb P^{r+h}_K$ be a maximal projecting surface
for $Y.$ If $B$ denotes the homogeneous coordinate ring for $Y,$ we have the inclusions
\[
A \hookrightarrow B = K[D(A)_1] \subseteq D(A) = \oplus_{n \in \mathbb Z} H^0(X, {\mathcal O}_X(n)),
\]
where $D(A) := \varinjlim \Hom_S((S_{+})^n, A)$ denotes the $S_{+}$-transform of $A.$ Observe that
\[
\dim_K D(A)_1 = r+h+1 \text{ and that } \deg Y = \deg X = d.
\]
As $Y \subset \mathbb P^{r+h}_K$ is non-degenerate
we have $d - \codim Y \geq 1,$ whence $d \geq r + h^1_{A}(1) -1.$
\end{reminder}

\begin{corollary} \label{5.13} Let $4 < r < d$ and assume that the surface $X \subset \mathbb P^r_K$ is of maximal
sectional regularity, of arithmetic depth one and Cohen-Macaulay. Then $d > 2r -5.$
\end{corollary}

\begin{proof}  By Reminder \ref{5.12} (B) we have $h^1_A(1) \leq d-r+1.$
By Remark \ref{5.4} (A) we have $\sigma (X) = 0$ and $\# \Sing(X) < \infty.$ So, by
Proposition \ref{4.9} (c) we have $h^2_A(0) = e(X).$ As $X$ is Cohen-Macaulay, we have
$e(X) = 0.$ Now, we conclude by Corollary \ref{5.11}
\end{proof}

\section{Surfaces with high linear deficiency}
We keep all notation and hypotheses of sections 4 and 5. According to Reminder \ref{5.12} (B) we have
$h^1_{A}(1) \leq d-r+1.$ In this section, we shall consider the situation in which
$h^1_{A}(1)$ is "close to being maximal", more precisely the cases in which
\[
h^1_{A}(1) \in \{ d-r-1, d-r, d-r+1\}.\]
We start by recalling a few facts on certain surfaces $Y \subset \mathbb P^s_K$ of degree $s+1.$ 

\begin{reminder} \label{6.1} (A) Let $s \geq 5$ and let $Y \subset \mathbb P^s_K$ be a non-degenerate 
and irreducible surface of degree $s+1$ with homogeneous coordinate ring $B.$ Assume that $Y$ is a maximal 
projecting surface of the surface $X \subset \mathbb P^r_K$ so that $r \leq s$ and $h^1_B(1) = 0.$ 
Then, according to \cite{B1} the surface $Y$ must be one of the six types I, IIA, IIA', IIIA, IVA0, IVA1 
which were introduced there. For these 6 types one has the following facts (cf. \cite[Propositions 4.2, 4.11, 5.5 and 5.6]{B1}): 
\begin{itemize}
\item[(i)] Type I: $Y$ is arithmetically Cohen-Macaulay, $\sreg Y = 3$ and $\sigma(Y) = 2.$
\item[(ii)] Type IIA: $Y$ is Cohen-Macaulay, $\depth B = 2, \sreg Y = 3, h^2_B(0) = 1, h^2_B(n) = 0$ 
for all $n \not= 0$ and $\sigma (Y) = 2.$ 
\item[(iii)] Type IIA': $\sreg Y = 3, \depth B = 2, h^2_B(n) = e(Y) = 1$ for all $n \leq 0, h^2_B(n) = 0$ 
for all $n > 0$ and $\sigma (Y) = 0.$ 
\item[(iv)] Type IIIA: $\sreg Y = 3, \depth B = 2, h^2_B(n) = e(Y) = 2$ for all $n \leq 0, h^2_B(n) = 0$ 
for all $n > 0$ and $\sigma (Y) = 0.$ 
\item[(v)] Type IVA1: $\sreg Y = 4, \depth B = 2, h^2_B(n) = e(Y) = 3$ for all $n \leq 0, h^2_B(1) = 1, h^2_B(n) = 0$ 
for all $n > 1$ and $\sigma (Y) = 0.$
\item[(vi)] Type IVA0: $\sreg Y = 4, \depth B = 1, h^1_B(1) = 1, h^1_B(n) = 0$ for all $n \not= 0$ and the values 
of $h^2_B(n), e(Y)$ and $\sigma(Y)$ are as in statement (iv).
\end{itemize}

(B) Again, let $s \geq 5$ and $Y \subset \mathbb P^s_K$ a non-degenerate reduced and irreducible surface of degree $s+1$ 
with homogeneous coordinate ring $B.$ Assume this time that $h^1_B(n) = 2$ for $n = 1,2$ and $h^1_B(n) = 0$ for 
$n \not= 1,2.$ Then, according to \cite{B1}, the surface $Y$ must be one of the types IIC or IVC which were introduced 
there. For these two types one has (cf. \cite[Propositions 5.5, 5.6]{B1}). 
\begin{itemize}
\item[(i)] Type IIIC: $Y$ is Cohen-Macaulay, $\sreg Y = 3, H^2(B) = 0$ and $\sigma (Y) = 0.$ 
\item[(ii)] Type IVC: $Y$ is Cohen-Macaulay, $\sreg Y = 4, H^2(B) = 0$ and $\sigma (Y) = 0.$
\end{itemize} 
In both cases we also have $h^3_B(n) = 0$ for all $n \geq -1.$ By our assumption $\e H^1(B) = 2,$ and so in both cases 
$\reg Y = 4.$ 
\end{reminder}

\begin{lemma} \label{6.2} Let $A := A.$ Let $ h \in \{0, \ldots, h^1_A(1)\}$ and
let $Y \subset \mathbb P^{r+h}_K$ be a projecting surface for $X.$ Let $B$ denote
the homogeneous coordinate ring of $Y.$ Then:
\begin{itemize}
  \item[(a)] $h^1_B(1) = h^1_A(1) -h.$
  \item[(b)] $h^i_B(n) = h^i_A(n)$ for all $i \geq 2$ and all $n \in \mathbb Z.$
  \item[(c)] $\sigma (Y) = \sigma (X).$
  \item[(d)] $e(Y) = e(X).$
  \item[(e)] $\sreg Y \leq \sreg X \leq d-r+3.$ 
\end{itemize}
\end{lemma}

\begin{proof} (a): We have the inclusions $A \hookrightarrow B \hookrightarrow D(A) = D(B),$ in which
$D(\cdot)$ denotes the formation of $S_{+}$-transform. It follows
\[
h^1_B(1) = \dim_K
D(B)_1 - \dim_K B_1 = \dim_K D(A)_1 - (\dim_K A_1 + h) = h^1_A(1) -h.\]

(b): As $D(A) = D(B)$ we have $H^i(B) = H^i(A)$ for all $i > 1.$

(c): Clear from statement (b) and Remark \ref{6.1}.

(d): Follows from $Y\simeq X.$

(e): In view of Remark \ref{5.2} (B) it is enough to show the first inequality. We write $C := B/A,$ 
choose $f \in S_1 \setminus \{0\}$ and consider the $S_+$-torsion modules $V := C/fC$ and 
$U := (0:_C f)(-1).$ As the multiplication maps $f : A(-1) \to A$ and $f : B(-1) \to B$ are injective 
the snake lemma yields a short exact sequence of graded $S$-modules 
\[
0 \to U \to A/fA \to B/fB \to V \to 0.
\]
As $U$ and $V$ are $S_+$-torsion we thus get an epimorphism $H^1(A/fA) \twoheadrightarrow H^1(B/fB)$ and 
an isomorphism $H^2(A/fA) \simeq H^2(B/fB).$ So $\e H^i(A/fA) \geq \e H^i(B/fB), i = 1,2.$ Whence 
$\reg \Proj(A/fA) \geq \reg \Proj(B/fB) \geq \sreg Y,$ as required.
\end{proof}

\begin{theorem} \label{6.3} Let $4 < r < d,$ set $h := h^1_A(1)$ and let
$Y \subset \mathbb P^{r+h}_K$ be a maximal projecting surface for $X.$
\begin{itemize}
  \item[(a)] If $h = d-r+1,$ then  $Y \subset \mathbb P^{d+1}_K$ is a smooth rational normal surface scroll,
  hence $X$ is smooth and rational with $H^2(A) = 0$ and $\sigma (X) = 0.$
  \item[(b)] If $h = d-r,$ then either
	\begin{itemize}
		\item[(i)] $Y \subset \mathbb P^d_K$ is a non-normal Del Pezzo
		  surface in the sense of \cite{BS4}, in particular $X$ is Cohen-Macaulay, $X\setminus \Nor(X)$ is a line,
                  $H^2(A) = 0$ and $\sigma (X) = 1,$ or
                \item[(ii)] $Y \subset \mathbb P^d_K$ is a surface of almost minimal degree of arithmetic depth 2,
                 in particular $\CM(X) = \Nor(X), X \setminus \Nor(X)$ consists of a single point, $h^2_A(n) = e(X) = 1$ for all $n \leq 0, h^2_A(n) = 0$ for all $n > 0,$ and $\sigma (X) = 0.$
         \end{itemize}
  \item[(c)] If $h = d-r-1,$ then $Y \subset \mathbb P^{d-1}_K$ is one of the six types (i) -- (vi) 
  listed in Reminder \ref{6.1} (A) and :
  \begin{itemize}
    \item[(i)] If $Y$ is of type I, then $X$ is Cohen-Macaulay, $H^2(A) = 0$ and $\sigma (X) = 2.$
    \item[(ii)] If $Y$ is of type II A, $X$ is Cohen-Macaulay, $h^2_A(0) = 1, h^2_A(n) = 0$ for all $n \not= 0$ and $\sigma (X) = 1.$
    \item[(iii)] If $Y$ is of type II A', $h^2_A(n) = e(X) = 1$ for all $ n \leq 0, h^2_A(n) = 0$ for all $n > 0$ and $\sigma (X) = 1.$
    \item[(iv)] If $Y$ is of type III A or IV A0, $h^2_A(n) = e(X) = 2$ for all $n \leq 0, h^2_A(n) = 0$ for all $n > 0$ and $\sigma (X) = 0.$
    \item[(v)] If $Y$ is of type IV A1, $h^2_A(n) = e(X) = 3$ for all $n \leq 0, h^2_A(1) = 1, h^2_A(n) = 0$ for all $n > 1$ and $\sigma (X) = 0.$
\end{itemize}
\end{itemize}
\end{theorem}

\begin{proof}
Let $B$ be the homogeneous coordinate ring of $Y.$ According to Lemma \ref{6.2} (a) we have $h^1_B(1) = 0.$

(a): Let $h = d-r+1.$ Then, $Y \subset \mathbb P^{d+1}_K$ is of degree $d$ and hence
a surface of minimal degree. As $d > 5, Y$ cannot be the Veronese surface and hence, by
the classification of varieties of minimal degree, is either a smooth scroll or a cone over a
rational normal curve (cf. \cite[Theorem 19.9]{H}) . In the latter case, each line in $\mathbb P^{d+1}_K$ passing through the
vertex of $Y$ would be tangent to $Y$ and thus no linear projection $\mathbb P^{d+1}_K \setminus \mathbb P^{d-r}
\twoheadrightarrow \mathbb P^r_K$ with $\mathbb P^{d-r}_K \cap Y = \emptyset$ could induce an isomorphism.
So, $Y$ is smooth and rational and hence so is $X.$

Moreover $Y$ is arithmetically Cohen-Macaulay, so that $H^2(B) = 0.$ In addition $\sigma (X) = 0.$
By Lemma \ref{6.2} we get $H^2(A) = 0$ and $\sigma (X) = 0.$

(b): Let $h = d-r.$ Then $Y \subset \mathbb P^d_K$ is of almost minimal degree in the sense
of \cite{BS4}. As $h^1_B(1) = 0,$ the surface $Y$ is linearly normal and hence of arithmetic
depth $t  \geq 2,$ (cf. \cite[Proposition 3.1]{BS4}).

(i): If $t = 3,$ then $Y$ is not normal and $Y \setminus \Nor(Y)$ is a line (cf. \cite[Theorems 1.4, 1.3]{BS4}).
So, $Y$ is a non-normal maximal Del Pezzo surface in the sense of \cite{BS4}. Clearly $H^2(B) = 0.$
According to \cite[Theorem 6.2]{BS4} we also have $h^3_B(-1) = 1$ and hence $h^3_B(0)= 0.$ So, by
Remark \ref{4.8} (D) we get $\sigma (Y) = 1.$ Now our claims on $X$ follow by Lemma \ref{6.2} (and the
fact that $Y$ and $X$ are isomorphic by means of a projection.

(ii): If $t = 2,$ \cite[Theorem 4.2]{BS4} yields that $h^2_B(n) = 1$ for all $n \leq 0, h^2_B(n) = 0$
for all $n > 0$ and $h^3_B(n) = 0$ for all $n \geq -1.$ so, by Remark \ref{4.8} (A) and (D) we get 
$e(Y) = 1$ and $\sigma (Y) = 0.$ Moreover, by \cite[Theorem 1.3]{BS4} we
have $\CM(Y) = \Nor(Y).$ Now we get our claims on $X$ again by Lemma \ref{6.2} and the isomorphism
$Y \stackrel{\simeq}{\longrightarrow} X.$

(c): Let $h = d-r-1.$ Then $Y \subset \mathbb P^{d-1}_K$ is of degree $d \geq 6.$ 
Now, on use of Reminder \ref{6.1} (A) and by Lemma \ref{6.2} one easily proves
our claims.
\end{proof}

\section{A few examples}
We keep the prevoius notation and hypotheses. The aim of this section is to give a few examples of surfaces which illustrate the results of sections 5 and 6. If $X \subset \mathbb P^r_K = \Proj S$ with
$S = K[x_0,\ldots, x_r],$ is a non-degenerate surface with homogeneous
vanishing ideal $I,$ homogeneous coordinate ring $A$ and arithmetic depth $t,$ the
{\sl Betti diagram} of $X$ is the diagram of size $(r-1-t)\times (\reg X - 1)$ whose
entry in the $i$-th column and the $j$-th row is given by
\[
\beta_{i,j} : = \dim_k \Tor^S_i(K,A)_{i+j}, 1 \leq i \leq r+1-t, 1 \leq j \leq \reg X -1.
\]
All the occuring computations were performed on use of {\sc Singular} \cite{GP}.

We first present an example of a surface $X \subset \mathbb P^r_K$ of degree $d$ which satisfies 
$\reg X = d-r+3$ but is not of maximal sectional regularity (cf. Remark \ref{5.4} (B)). We also 
show that this surface $X$ has a reduced and irreducible hyperplane section curve $\mathcal D$ 
with $\reg {\mathcal D} = \reg X > \sreg X,$ (cf. Remark \ref{5.4} (C)). 

\begin{example} \label{7.1} We project the smooth rational surface scroll $Y := S(2,5)
\subset \mathbb P^8_K,$ which is given by the $2\times 2$-minors of the matrix
\[
\left(
\begin{array}{cc|ccccc}
x_0 & x_1 & x_3 & x_4 & x_5 & x_6 & x_7 \\
x_1 & x_2 & x_4 & x_5 & x_6 & x_7 & x_8
\end{array}
\right)
\]
from the line $L = \mathbb P^1_K \subset \mathbb P^8_K \setminus \operatorname{Sec}(Y)$
defined by $x_0 = x_1 = \ldots = x_4 = x_7 = x_8 = 0.$ We get a non-degenerate smooth
irreducible surface $X \subset \mathbb P^6_K$ of degree 7 which is of arithmetic depth 1,
of regularity 4 and has the Betti diagram
\[
\begin{array}{|c|cccccc|}
\hline
  & 1 & 2 & 3 & 4 & 5 & 6 \\
\hline
1 & 6 & 8 & 3 & 0 & 0 & 0 \\
2 & 4 &12 & 12& 4 & 0 & 0 \\
3 & 4 &18 & 32& 28& 12& 2 \\
\hline
\end{array}
\]
Computing $h^1_A(n) = \dim_K \Ext_S^6(A, S(-7))_n$ yields $h^1_A(n) = 2$ for $n = 1,2$
and $h^1_A(n) = 0$ for all $n \in \mathbb Z \setminus \{1,2\}.$ So, the surface $X$
must satisfy $\sreg X \in \{3,4\}$ (cf. Reminder \ref{6.1} (B) (ii), (ii)). 
In particular, we have $2 = \delta (X) = \e H^1(A) = \reg X -2$ (cf. Reminder \ref{6.1} 
(B) (iii)).

We write $S := K[x_0,x_1,\ldots,x_4,x_7,x_8]$ and $f := x_0-x_1-x_2-x_3-x_4-x_7-x_8$
and consider the hyperplane section ${\mathcal C} = X \cap \Proj (S/fS) \subset \Proj (S/fS) =
\mathbb P^5_K \subset \mathbb P^6_K$ of $X.$ Computing the primary decomposition of
$J := (I + fS)^{\sat}/fS \subset S/fS =: T$ we find that $J \in \Spec T$ so that
${\mathcal C} = \Proj(T/J) \subset \Proj T = \mathbb P^6_K$ is a non-degenerate irreducible curve.
The Betti diagram of the curve $Z$ presents itself as follows
\[
\begin{array}{|c|ccccc|}
\hline
  & 1 & 2 & 3 & 4 & 5 \\
\hline
1 & 6 & 8 & 3 & 0 & 0 \\
2 & 6 & 20& 24& 12& 2 \\
\hline
\end{array}
\]
In particular the hyperplane section $\mathcal C$ of $X$ satisfies $\reg {\mathcal C} = 3.$ Therefore, 
$\sreg X = 3 < 4 = \reg X = \deg X -6+3.$ So $X$
is not of maximal sectional regularity, whereas $\reg X$ takes the conjectured maximal value 4.

Let $X$ and $S$ be as before and let $g := x_0-x_1-x_8.$ Consider the hyperplane 
section ${\mathcal D} = \Proj A/gA = X \cap \Proj S/gS \subset \Proj S/gS$ 
of $X.$ Again by computing the primary decomposition of $(I + gS)^{\sat}/gS \subset S/gS$ we see that 
${\mathcal D} \subset \mathbb P^5_K$ is a non-degenerate reduced and irreducible curve. The Betti diagram 
of $\mathcal D$ is given by 
\[
\begin{array}{|c|ccccc|}
\hline
  & 1 & 2 & 3 & 4 & 5  \\
\hline
1 & 7 & 8 & 3 & 0 & 0  \\
2 & 0 & 6 & 8 & 3 & 0 \\
3 & 1 & 4 & 6 & 4 & 1 \\
\hline
\end{array}
\]
In particular we have indeed $\reg {\mathcal D} = 4 > \sreg X = 3.$
\end{example}

We now give a number of examples which illustrate Theorem \ref{6.3}.

\begin{example} \label{7.2} We project the smooth rational surface scroll $Y = S(1,8) \subset
\mathbb P^{10}_K$ given by the $2\times 2$-minors of the matrix
\[
\left(
\begin{array}{c|cccccccc}
x_0 & x_2 & x_3 & x_4 & x_5 & x_6 & x_7 & x_8 & x_9 \\
x_1 & x_3 & x_4 & x_5 & x_6 & x_7 & x_8 & x_9 & x_{10}
\end{array}
\right)
\]
from the $3$-space $\mathbb P^3_K \subset \mathbb P^{10}_K \setminus \operatorname{Sec} Y$
given by $x_0= \ldots = x_4= x_9 = x_{10} = 0.$ We get a non-degenerate smooth irreducible
surface $X \subset \mathbb P^6_K$ of degree 9 which is of arithmetic depth 1, of regularity
6 and has the Betti diagram
\[
\begin{array}{|c|cccccc|}
\hline
  & 1 & 2 & 3 & 4 & 5 & 6 \\
\hline
1 & 6 & 8 & 3 & 0 & 0 & 0 \\
2 & 0 & 0 & 0 & 0 & 0 & 0 \\
3 & 4 & 12& 12& 4 & 0 & 0 \\
4 & 4 & 18& 32& 28& 12& 2 \\
5 & 6 & 28& 52& 48& 22& 4 \\
\hline
\end{array}
\]
A computation of $h^1_A(n)$ gives the following table for the non-vanishing values of the
Hartshorne-Rao function
\[
\begin{array}{|c|cccc|}
\hline
n & 1 & 2 & 3 & 4 \\
\hline
h^1_A(n) & 4 & 8 & 8 & 4 \\
\hline
\end{array}
\]
In particular $\delta (X) = 3 < \e H^1(A) = 4.$ Moreover $h^1_A(1) = 4 = \deg X - 6 +1,$ so that
$X \subset \mathbb P^6_K$ is of type (a) of Theorem \ref{6.3}. In accordance with Theorem
\ref{6.3} the projecting surface $Y \subset \mathbb P^9_K$ is a smooth rational surface
scroll and computing $H^2(A) \simeq \Hom_A(\Ext^5_S(A, S(-7)), K)$ confirms that $H^2(A)$
vanishes, as predicted.

Now, let $f \in S_1 \setminus \{0\}.$ Then, the induced exact sequence
\[
0 \to H^0(A/fA)_n \to H^1(A)_{n-1} \stackrel{f}{\to} H^1(A)_n \to H^1(A/fA)_n \to 0
\]
and the above table for the values of $h^1_A(n)$ imply that $h^1_{A/fA}(1) = 4$ and
$h^1_{A/fA}(2) \geq 4.$ This shows, that statement (c) of Proposition \ref{4.3} need not
hold if $d = 2r -3.$ So, concerning statement (c), the bound on the degree required
in Proposition \ref{4.3} is sharp.

According to Bertini there is a (unique maximal) dense open set $U \subset S_1 \setminus \{0\}$
such that ${\mathcal C} := \Proj A/fA = X \cap \Proj S/fS \subset \Proj S/fS = \mathbb P^5_K$
is a non-degenerate reduced irreducible curve of degree 9 for all $f \in U.$ If for some
$f \in U$ the curve ${\mathcal C} = \Proj A/fA$ is of maximal regularity (that is of
regularity 6) the inequality $h^1_{A/fA}(2) \geq 4 > 3 = \deg {\mathcal C} -5 -1$ implies
(cf. \cite[Theorem 3.3]{BS2}) that the union of $\mathcal C$ with an extremal secant line
$\mathbb L = \mathbb P^1_K$ of $\mathcal C$ is never arithmetically Cohen-Macaulay.
We write $S = K[x_0,\dots , x_4,x_9,x_{10}]$ and choose $f := x_1 - x_2.$ Then, computing
the primary decomposition of $J := (I+fS)^{\sat}/fS \subset S/fS =: T$ we see that $J \in
\Spec T$ so that the hyperplane section ${\mathcal C} = \Proj T/J = X \cap \Proj T \subset
\Proj T = \mathbb P^5_K$ is a non-degenerate reduced irreducible curve of degree 9 of arithmetic  depth 1 satisfying $\reg {\mathcal C} = 6.$ The Betti diagram of $\mathcal C$ is
given by
\[
\begin{array}{|c|ccccc|}
\hline
  & 1 & 2 & 3 & 4 & 5 \\
\hline
1 & 6 & 8 & 3 & 0 & 0 \\
2 & 2 & 4 & 0 & 0 & 0 \\
3 & 1 & 4 & 10& 6 & 1 \\
4 & 0 & 0 & 0 & 0 & 0 \\
5 & 1 & 4 & 6 & 4 & 1 \\
\hline
\end{array}
\]
Clearly $\mathcal C$ is a curve of maximal regularity. By the previous observation
for each secant line $\mathbb L$ of $\mathcal C$ the union
${\mathcal C} \cup \mathbb L$ is not arithmetically Cohen-Macaulay.
\end{example}

The next two examples shall illustrate the statements of Theorem \ref{6.3} (b) (i) and (ii).

\begin{example} \label{7.3} (A) We start with the same scroll $W = S(1,8) \subset \mathbb P^{10}_K$ as in the previous example. We project from the point $p = (0:1:1:0: \ldots :0) \in
\mathbb P^{10}_K \setminus W$ by means of the map $(x_0:x_1: \ldots : x_{10}) \mapsto
(x_0:x_1 - x_2: x_3: \ldots :x_{10}).$ We then get a non-degenerate reduced irreducible surface
$Y \subset \mathbb P^9_K$ of degree 9 whose Betti diagram has the shape
\[
\begin{array}{|c|ccccccc|}
\hline
  & 1 & 2 & 3 & 4 & 5 & 6 & 7 \\
\hline
1 & 27&105&189&189&105& 27& 0\\
2 & 0 & 0 & 0 & 0 & 0 & 0 & 1 \\
\hline
\end{array}
\]
So $Y \subset \mathbb P^9_K$ is a surface of almost minimal degree which is arithmetically
Cohen-Macaulay. So $Y$ is maximally Del Pezzo (cf. \cite[Theorem 7.2]{BS4}) and not
normal (cf. \cite[Theorem 1.3 (a)]{BS4}). After having introduced new coordinates, we 
now canonically project $Y \subset \mathbb P^9_K
= \Proj S, S = K[y_0,\ldots, y_9],$ from the plane $\mathbb P^2_K \subset \mathbb P^9_K
\setminus Y$ given by $y_0 = y_1 = y_2 = y_3 = y_4 = y_8 = y_9 = 0.$ What we get is a non-degenerate reduced irreducible surface $X \subset \mathbb P^6_K$ of degree 9 with the following
Betti diagram
\[
\begin{array}{|c|cccccc|}
\hline
  & 1 & 2 & 3 & 4 & 5 & 6 \\
\hline
1 & 4 & 2 & 0 & 0 & 0 & 0 \\
2 & 6 & 21& 20& 6 & 0 & 0 \\
3 & 0 & 0 & 0 & 0 & 0 & 0 \\
4 & 5 & 23& 42& 38& 17& 3\\
\hline
\end{array}
\]
So, $X$ is of arithmetic depth 1 and of regularity 5. Moreover, the non-vanishing values
of the Hartshorne-Rao function are computed as follows
\[
h^1_A(1) = 3, h^1_A(2) = 4, h^1_A(3) = 3.
\]
As $h^1_A(1) = 3 = \deg X - 6$ we are in the situation of statement (b) of Theorem
\ref{6.3} and the morphism $Y \to X$ induced by our projection is an isomorphism. So,
$Y$ is a maximal projecting surface for $X$ which is non-normal Del Pezzo. So, we actually
are in the case (b) (i) of Theorem \ref{6.3}.

As $\reg X = 5 < \deg X -6 + 3,$ the surface $X$ cannot be of maximal sectional regularity
(cf. Remark \ref{5.2} (B)). To illustrate this directly, we consider the hyperplane section curve ${\mathcal C} := \Proj A/fA
\subset \Proj S/fS$ with $f := x_1 - x_5 + x_{10}.$ Computing the primary
decomposition of the vanishing ideal $J \subset S/fS$ of $\mathcal C$ we see that this
curve is indeed reduced and irreducible. The Betti diagram of $\mathcal C$ is computed to be
\[
\begin{array}{|c|ccccc|}
\hline
  & 1 & 2 & 3 & 4 & 5 \\
\hline
1 & 5 & 2 & 0 & 0 & 0 \\
2 & 2 & 15& 16& 5 & 0 \\
3 & 0 & 0 & 0 & 0 & 0 \\
4 & 1 & 4 & 6 & 4 & 1 \\
\hline
\end{array}
\]
So we have indeed $\reg {\mathcal C} = 5 < \deg {\mathcal C} - 5 + 2.$

(B) As in part (A) we start with the smooth surface scroll $W = S(1,8) \subset \mathbb P^{10}_K.$ We project $W$ canonically from the point $q = (0:0:0:1:0: \ldots :0) \in \mathbb
P^{10}_K \setminus W$ and get again a surface $Y \subset \mathbb P^9_K = \Proj K[x_0,x_1,x_2,x_4, \ldots, x_{10}]$ of degree 9
whose Betti diagram has the shape
\[
\begin{array}{|c|cccccccc|}
\hline
  & 1 & 2 & 3 & 4 & 5 & 6 & 7 & 8 \\
\hline
1 & 26& 98&168&154& 70& 6 & 0 & 0 \\
2 & 1 & 7 & 21& 35& 35& 28& 9 & 1 \\
\hline
\end{array}
\]
So, $Y$ is a surface of almost minimal degree and arithmetic depth 2.

We project $Y$ again canonically from the plane $\mathbb P^2_K \subset \mathbb P^9_K \setminus
Y$ given by $x_0 = x_1 = x_2 = x_4 = x_5 = x_9 = x_{10}$ and get a surface $X \subset \mathbb
P^6_K$ of degree 9 whose Betti diagram is as follows
\[
\begin{array}{|c|cccccc|}
\hline
  & 1 & 2 & 3 & 4 & 5 & 6 \\
\hline
1 & 3 & 2 & 0 & 0 & 0 & 0 \\
2 & 11& 31& 30& 11& 1 & 0 \\
3 & 0 & 0 & 0 & 0 & 0 & 0 \\
4 & 5 & 23& 42& 38& 17& 3 \\
\hline
\end{array}
\]
In particular $X$ is of arithmetic depth 1 and we have $\reg X = 5 < \deg X - 6 + 3.$ The
non-vanishing values of the Hartshorne-Rao function are the same as in the example
presented in part (A), so that in particular $\delta (X) = 2 < \e H^1(A)$ and $h^1_A(n) =
\deg X - 6.$ The last equation shows again that $Y$ is a projecting surface for $X$ and hence we are in the case of statement (b)(ii) of Theorem \ref{6.3}. A computation
confirms that $h^2_A(n) = \dim_K \Ext_S^5(A, S(-7))_{-n}$ equals 1 for all $b \leq 0$ and
vanishes for all $n > 0.$ As $\reg X = 5 < \deg X -6 + 3,$ the surface $X \subset \mathbb P^6_K$ cannot be of maximal sectional regularity (cf. Remark \ref{5.2} (B)). This can be confirmed
directly by considering the hyperplane section curve ${\mathcal C} = \Proj A/fA \subset
\Proj S/fS = \mathbb P^5_K$ with $f = x_1-x_2,$ which is verified
to be reduced and irreducible and has the Betti diagram
\[
\begin{array}{|c|ccccc|}
\hline
  & 1 & 2 & 3 & 4 & 5 \\
\hline
1 & 4 & 2 & 0 & 0 & 0 \\
2 & 7 & 25& 26& 10& 1 \\
3 & 0 & 0 & 0 & 0 & 0 \\
4 & 1 & 4 & 6 & 4 & 1 \\
\hline
\end{array}
\]
So indeed $\reg {\mathcal C} = 5 < \deg {\mathcal C} - 5 + 2.$ Observe that the necessary
condition "$d > 2r -5$" of Theorem \ref{6.3} is satisfied here. So, the present
example shows that (in the case (b) (ii) of Theorem \ref{6.3}) this condition is not sufficient to guarantee that $X$ is of maximal sectional regularity.
\end{example}

We now present a few examples which fall under the cases (c) of Theorem \ref{6.3}.

\begin{example} \label{7.4} (A) Let $W \subset \mathbb P^9_K$ denote the non-normal
Del Pezzo surface considered in Example \ref{7.3} (A). We rename our indeterminates
and write $\mathbb P^9_K = \Proj (K[x_0,\ldots, x_9]).$ We project $W$ from the point
$p := (1:0:\ldots:0:1)$ by means of the map $(x_0:\ldots :x_9) \mapsto (x_0-x_9:x_1 : \ldots : x_8)$ and get a surface $Y \subset \mathbb P^8_K$ of degree 9 whose Betti diagram is
\[
\begin{array}{|c|cccccc|}
\hline
  & 1 & 2 & 3 & 4 & 5 & 6 \\
\hline
1 & 19& 58& 75& 44& 5 & 0 \\
2 & 0 & 0 & 0 & 0 & 6 & 2 \\
\hline
\end{array}
\]
In particular $Y$ now is arithmetically Cohen-Macaulay. Moreover the canonical module
$K(A) = K^3(A) = \Ext^6_S(A,S(-9))$ is computed to be minimally generated by 2
homogeneous elements of degree 1. So, in particular $Y$ must be of type I (cf. Reminder 
\ref{6.1} (A)).

Now, we project $Y$ from the line $\mathbb P^1_K \subset \mathbb P^8_K \setminus Y$ given by
$x_0 = x_1 = x_2 = x_3 = x_6 = x_7 = x_8 = 0$ and get a surface $X \subset \mathbb P^6_K$ of degree 9 with the
diagram
\[
\begin{array}{|c|cccccc|}
\hline
  & 1 & 2 & 3 & 4 & 5 & 6 \\
\hline
1 & 3 & 2 & 0 & 0 & 0 & 0 \\
2 & 13& 39& 42& 19& 3 & 0 \\
3 & 1 & 5 & 10& 10& 5 & 1 \\
\hline
\end{array}
\]
The non-vanishing values of the Hartshorne-Rao function are computed to be $h^1_A(1) = 2,
h^1_A(2) = 2$ and $H^2(A)$ turns out to be 0. So, we must be in the situation described
in statement (c) (i) of Theorem \ref{6.3}. Again we have $\delta(X) = \e H^1(A) = 2 < 6 =
\deg X -6 + 3$ and $\reg X = 4$ so that $X$ surely cannot be of maximal sectional regularity.

Indeed, by computation it turns out that the hyperplane section curve ${\mathcal C} =
\Proj A/fA$ of $X$ is reduced and irreducible for $f = x_0-x_6$ and $f = x_1-x_2.$ It is
interesting to note that in the first case $\reg {\mathcal C}$ takes the maximally possible
value $4 = \reg X,$ where as in the second case we have $\reg {\mathcal C} = 3.$

(B) Our next example is of the type mentioned under statement \ref{6.3} (c) (ii). We first
project the scroll $W = S(3,6) \subset \mathbb P^{10}_K = \Proj K[x_0,\ldots,x_{10}]$
given by the $2\times 2$-minors of the matrix
\[
\left(
\begin{array}{ccc|cccccc}
x_0 & x_1 & x_2 & x_4 &x_5 & x_6 & x_7 & x_8 & x_9 \\
x_1 & x_2 & x_3 & x_5 &x_6 & x_7 & x_8 & x_9 & x_{10}
\end{array}
\right)
\]
canonically from the line $\mathbb P^1_K \subset \mathbb P^{10}_K \setminus W$ defined by
$x_0 = x_1 = x_3 = x_4 = \ldots = x_8 = x_{10} = 0.$ We get a surface $Y \subset \mathbb P^8_K =
\Proj (K[x_0,x_1,x_3, \ldots, x_8,x_{10}])$ of degree 9, arithmetic depth 2, regularity 3 and having the Betti diagram
\[
\begin{array}{|c|ccccccc|}
\hline
  & 1 & 2 & 3 & 4 & 5 & 6 & 7 \\
\hline
1 & 18& 52& 60& 24& 0 & 0 & 0 \\
2 & 1 & 6 & 15& 30& 27& 9 & 1 \\
\hline
\end{array}
\]
In particular $Y$ can only be of type IIA, IIA' or III A (cf. Reminder \ref{6.1} (A)). We now project $Y$ canonically
from the line $\mathbb P^1_K \subset \mathbb P^8_K \setminus Y$ given by $x_0 - x_8 = x_1 =
x_3 = x_4 = x_5 = x_7 = x_{10}.$ We obtain a surface $X \subset \mathbb P^6_K$ of degree
9 with Betti diagram
\[
\begin{array}{|c|cccccc|}
\hline
  & 1 & 2 & 3 & 4 & 5 & 6 \\
\hline
1 & 3 & 0 & 0 & 0 & 0 & 0 \\
2 & 9 & 30& 27& 8 & 0 & 0 \\
3 & 3 & 15& 29& 27& 12& 2 \\
\hline
\end{array}
\]
In particular $X$ is of arithmetic depth 1 and satisfies $\reg X = 4.$ The non-vanishing
values of the Hartshorne-Rao function of $X$ are computed to be $h^1_A(n) = 2$ for $n = 1,2.$
In particular $h^1_A(1) = 2 = \deg X - 6 -1,$ so that $Y \subset \mathbb P^8_K$ is a
maximal projecting surface for $X$ and we must be in one of the cases mentioned in
statement (c) of Theorem \ref{6.3}. Another computation yields that $h^2_A(0) = 1$ and
$h^2_A(n) = 0$ for all $n \not= 0.$ So, we must be in the case (c) (ii). If $f \in
S_1 \setminus \{0\}$ is a linear form such that the hyperplane section curve
${\mathcal C} = \Proj A/fA$ of $X$ is reduced and irreducible, we must have $3 \leq
\reg {\mathcal C} \leq 4.$ Choosing $f := x_3 - x_4$ we get indeed a reduced irreducible
hyperplane section curve $\mathcal C$ of regularity 4.

(C) Next, we present an example for the case (c) (iii) of Theorem \ref{6.3}. Again we 
start with the scroll $W = S(1,8) \subset \mathbb P^{10}_K$ and project $W$ from the line 
$\mathbb P^1_K \subset \mathbb P^{10}_K \setminus W$ given by $x_0 = x_1-x_2 = x_3 = x_5 = \ldots = 
x_{10} = 0.$ We thus get a surface $Y \subset \mathbb P^8_K$ of degree nine whose 
Betti diagram is the same as the Betti diagram of the surface $Y$ in part (B).

So, the homogeneous coordinate ring $B$ of $Y$ has depth 2. Moreover, $H^2(B)$ is calculated 
to be isomorphic to the Matlis dual of $K[x_0].$ So, $Y$ falls under the type IIA' mentioned 
in Reminder \ref{6.1} (A) (iii). Now, we project $Y$ from the point $(0:\ldots:0:1:0:0) \in 
\mathbb P^8_K \setminus Y$ and get a surface $X \subset \mathbb P^7_K$ of degree 9 with Betti diagram 
\[
\begin{array}{|c|ccccccc|}
\hline
  & 1 & 2 & 3 & 4 & 5 & 6 & 7 \\
\hline
1 & 9 & 11 & 0 & 0 & 0 & 0 & 0 \\
2 & 5 & 36 & 81 & 75 & 36 & 9 & 1 \\
\hline
\end{array}
\]
whose homogeneous coordinate ring $A$ satisfies $H^1(A) \simeq K(-1).$ Now clearly we are 
in the requested case with $d = 9, r = 7$ and $h = 1.$ As $\sreg Y = 3$ by Reminder \ref{6.1} 
(A) (iii) we must have $\sreg X \geq 3$ (cf. Lemma \ref{6.2} (e)). For various linear forms,  
for instance $f = x_0-x_1-x_5-x_6,$ we computed $\reg \Proj (A/fA) = 3,$ whence $\sreg X = 3.$ 

(D) We now present an example for the case (c) (iv) of Theorem \ref{6.3}. We first project
the scroll $W = S(3,6) \subset \mathbb P^{10}_K$ of part (B) from the line $\mathbb P^1_K
\subset \mathbb P^{10}_K \setminus W$ given by $x_0 = x_2 = \ldots = x_8 = x_{10} = 0.$
We then get a surface $Y \subset \mathbb P^8_K = \Proj( K[x_0,x_2, \ldots, x_8,x_{10}])$
of degree 9 whose Betti diagram is given by
\[
\begin{array}{|c|ccccccc|}
\hline
  & 1 & 2 & 3 & 4 & 5 & 6 & 7 \\
\hline
1 & 17& 46& 45& 8 & 0 & 0 & 0 \\
2 & 2 & 12& 34& 65& 48& 16& 2 \\
\hline
\end{array}
\]
In particular, $Y$ is of arithmetic depth 2 and satisfies $\reg Y = 3.$ So $Y$ must be of
type IIA, IIA' or IIIA (cf. Reminder \ref{6.1} (A)). We now project $Y$ canonically from 
the point $p = (0:0:0:0:0:1:0:0:0)\in \mathbb P^8_K \setminus Y$ and get a surface 
$X \subset \mathbb P^7_K$ of degree 9 and regularity 3 with Betti diagram
\[
\begin{array}{|c|ccccccc|}
\hline
  & 1 & 2 & 3 & 4 & 5 & 6 & 7 \\
\hline
1 & 8 & 12& 3 & 0 & 0 & 0 & 0 \\
2 & 12& 54&101& 90& 42& 10& 1 \\
\hline
\end{array}
\]
In particular, $X$ is of arithmetic depth 1. Moreover by computation we get $H^1(A) \simeq
K(-1)$ and $h^2_A(n) = 2$ for all $n \leq 0.$ This shows that $Y$ is a maximal projecting
surface for $X$ and that we are in the case (c)(iv) of Theorem \ref{6.3}. The intersection
curve ${\mathcal C} = \Proj (A/(x_0-x_{10})A)$ is computed to be reduced and irreducible and 
to satisfy $\reg {\mathcal C} = 3.$ As $X$ is of arithmetic depth 1, each hyperplane section 
of $X$ must be of regularity $\geq 3.$ So, the generic hyperplane section of $X$ must be a 
curve of regularity 3. In particular $X$ cannot be of maximal sectional regularity. 

(E) Finally we present an example for the case (c) (v) of Theorem \ref{6.3}. This time we project 
the scroll $W = S(1,8) \subset \mathbb P^{10}_K$ from the line $\mathbb P^1_K \subset \mathbb P^{10} 
\setminus W$ given by $x_0 = x_1 = x_2 = x_5 = \ldots = x_{10} = 0$ and get a surface $Y \subset 
\mathbb P^8_K$ of degree 9 with Betti diagram
\[
\begin{array}{|c|ccccccc|}
\hline
  & 1 & 2 & 3 & 4 & 5 & 6 & 7 \\
\hline
1 & 18& 52& 60& 24& 5 & 0 & 0 \\
2 & 0 & 0 & 0 & 15& 12& 3 & 0 \\
3 & 1 & 6 & 15& 20& 15& 6 & 1 \\
\hline
\end{array}
\]
So the homogeneous coordinate ring $B$ of $Y$ has depth 2. Moreover a computation furnishes 
that $H^2(B) \simeq (K[x_0,x_1,x_2]/(x_0,x_1)^2){\check{}}(-1).$ So, $Y$ must be of type IVA1 
of Reminder \ref{6.1} (A) (v). Now, we project $Y$ from the point $(0:\ldots:0:1:1:0:0) 
\in \mathbb P^8_K \setminus Y,$ and get a surface $X \subset \mathbb P^7_K$ of degree nine 
with Betti diagram 
\[
\begin{array}{|c|ccccccc|}
\hline
  & 1 & 2 & 3 & 4 & 5 & 6 & 7 \\
\hline
1 & 9& 12& 3& 0& 0 & 0 & 0 \\
2 & 5 & 34 & 71 & 65& 31& 8 & 1 \\
3 & 1 & 5 & 10& 10& 5& 1 & 0 \\
\hline
\end{array}
\]
In particular the homogeneous coordinate ring $A$ of $X$ has depth 1, and a further computation 
shows that $H^1(A) \simeq K(-1).$ So, this time we are in the case (c) (v) of Theorem \ref{6.3} 
with $d = 9, r = 7$ and $h = 1.$ Again by Lemma \ref{6.2} (e), Reminder \ref{6.1} (A) (v) and 
by computing $\reg \Proj (A/fA)$ for some linear forms $f,$ e.g. $f = x_0-x_2+x_{10},$ we obtain that $\sreg X = 4.$ 

We now return to the surface $Y \subset \mathbb P^8_K.$ The Betti diagram of $Y$ tells us, that 
the homogeneous vanishing ideal $J \subset S = K[x_0,x_1,x_2,x_5,\ldots,x_{10}]$ of $Y$ is generated by 18 
quadrics and one quartic $Q.$ A {\sc Singular} computation gives $Q = x_1^3x_2-x_0^3x_5$ and 
$L := J_2S :_S Q = (x_5,x_6,x_7,x_8,x_9,x_{10}).$ So $\mathbb E := \Proj S/L = \mathbb P^2_K$ 
is a plane and $J + L = L + QS$ tells us that $Y \cap \mathbb E$ is the quartic defined in $\mathbb E$ 
by $Q.$ So, $Y$ admits a whole plane of $4 = (\deg Y -8 +3)$-secant lines, which is in accordance with the 
fact that $Y$ is of maximal sectional regularity (cf. Corollaries \ref{5.6} and \ref{5.7}). 
\end{example}

\begin{remark} \label{7.5} The {\sc Singular} files of the examples in this section are available 
upon request by the authors.
\end{remark}

\end{document}